\documentclass[10pt, english]{elsarticle}
\usepackage{amsthm}
\usepackage{amsmath}
\usepackage{latexsym, amssymb}
\usepackage{txfonts}
\usepackage{mathtools}
\usepackage{color}
\usepackage{babel}
\usepackage[all]{xy}
\usepackage[capitalise]{cleveref}

\newtheorem{thm}{Theorem}[section] 

\newtheorem{cor}[thm]{Corollary}

\newtheorem{lem}[thm]{Lemma}
\newtheorem{prop}[thm]{Proposition}

\theoremstyle{definition}

\newcommand\operA[2]{{\if!#2!\operatorname{#1}\else{\operatorname{#1}_{#2}^{\phantom{I}}}\fi}} 

\def\dim{{\operatorname{dim}}}

\newcommand{\Trace}[1][]{\if!#1!\operatorname{Tr}\else{\operatorname{Tr}_{#1}^{\phantom{I}}}\fi} 

\long\def\forget#1\forgotten{{}} %

\def\({\left(}
\def\){\right)}

\newcommand\LAY[3][]{{\begin{array}{c}\mbox{#2} \if#1!{}\else{+}\fi \\ \mbox{#3}\end{array}}}

\makeatletter
\newcommand{\bigperp}{%
  \mathop{\mathpalette\bigp@rp\relax}%
  \displaylimits
}

\newcommand{\bigp@rp}[2]{%
  \vcenter{
    \m@th\hbox{\scalebox{\ifx#1\displaystyle2.1\else1.5\fi}{$#1\perp$}}
  }%
}
\makeatother

\renewcommand{\geq}{\geqslant}

\makeatletter
\def\ps@pprintTitle{%
 \let\@oddhead\@empty
 \let\@evenhead\@empty
 \def\@oddfoot{\centerline{\thepage}}%
 \let\@evenfoot\@oddfoot}
\makeatother

\newif\iffurther
\furtherfalse

\input xy
\input xyidioms.tex
\usepackage{xy}
\xyoption{all} %

\begin{document}
\begin{frontmatter}

\title{Field of Iterated Laurent Series and its Brauer Group}

\author{Adam Chapman}
\ead{adam1chapman@yahoo.com}
\address{Department of Computer Science, Tel-Hai Academic College, Upper Galilee, 12208 Israel}

\begin{abstract}
The symbol length of ${_pBr}(k(\!(\alpha_1)\!)\dots(\!(\alpha_n)\!))$ for an algebraically closed field $k$ of $\operatorname{char}(k) \neq p$ is known to be $\lfloor \frac{n}{2}  \rfloor$. We prove that the symbol length for the case of $\operatorname{char}(k) = p$ is rather $n-1$. We also show that pairs of anisotropic quadratic or bilinear $n$-fold Pfister forms over this field need not share an $(n-1)$-fold factor.
 \end{abstract}

\end{frontmatter}

\section{Introduction}

Given a prime integer $p$ and a field $F$ of $\operatorname{char}(F) \neq p$ containing a primitive root of unity $\rho$, every cyclic algebra of degree $p$ over $F$ takes the form
$$(\alpha,\beta)_{p,F}=F \langle x,y : x^p=\alpha, y^p=\beta, yx=\rho xy \rangle$$
for some $\alpha,\beta \in F^\times$.
When $\operatorname{char}(F)=p$, every such an algebra takes the form
$$[\alpha,\beta)_{p,F}=F \langle x,y : x^p-x=\alpha, y^p=\beta, y x y^{-1}=x+1 \rangle$$
for some $\alpha \in F$ and $\beta \in F^\times$.
These forms are called ``(Hilbert) symbol presentations" of the algebras, and the algebras are also called ``symbol algebras".
By \cite{MS} (when $\operatorname{char}(F) \neq p$ and $\rho \in F$) and \cite[Chapter 7, Theorem 30]{Albert:1968} (when $\operatorname{char}(F)=p$), every class in $\prescript{}{p}Br(F)$ is represented by a a tensor product of symbol algebras of degree $p$.
The symbol length of a class is the minimal number of symbol algebras required in order to express it, and the symbol length of ${_pBr}(F)$ is the supremum on the symbol lengths of its classes.
Fields $F$ for which $\prescript{}{2}Br(F)$ has symbol length 1 are called ``linked fields" and are of interest in number theory because of their special arithmetic properties (for example, their $u$-invariant is either 0,1,2,4 or 8 by \cite{ElmanLam:1973} and \cite{ChapmanDolphin:2017}).

For complete discretely valued fields, the computation of the symbol length is often easy. In particular, the symbol length of $\prescript{}{p}Br(F)$ is $\lfloor \frac{n}{2} \rfloor$ for $F=k(\!(\alpha_1)\!)\dots(\!(\alpha_n)\!)$ being the field of iterated Laurent series in $n$ indeterminates over an algebraically closed field $k$ of $\operatorname{char}(k)\neq p$.
The goal of this paper is to show that this does not extend to the case of $\operatorname{char}(k)=p$, and in fact the symbol length of $\prescript{}{p}Br(F)$ is exactly $n-1$. 
We also study pairs of quadratic (or bilinear) Pfister forms over such fields when $p=2$ and show that they need not share an $(n-1)$-fold factor.

\section{Notation and Terminology}

We denote the symbol length of ${_pBr}(F)$ by $\operatorname{SymL}_p(F)$.
Given a valued division algebra $D$ (that can be a field, in particular), we denote its residue field by $\overline{D}$ and its value group by $\Gamma_D$. We refer the reader to \cite{TignolWadsworth:2015} for background on valuation theory on division algebras.

Given a field $F$ of $\operatorname{char}(F)=p>0$, if $[F:F^p]<\infty$ then $[F:F^p]=p^r$ for some nonnegative integer $r$. This $r$ is called the ``$p$-rank" of $F$, and denoted by $\operatorname{rank}_p(F)$.
We say that $I_q^n F$ is linked if every two quadratic $n$-fold Pfister forms over $F$ share an $(n-1)$-fold Pfister factor. In particular, $F$ is linked when $I_q^2 F$ is linked. We say that $I^n F$ is $m$-linked when every $m$ anisotropic bilinear $n$-fold Pfister forms over $F$ share an $(n-1)$-fold factor. When $m=2$ we simply say $I^n F$ is linked. When $\operatorname{char}(F)\neq 2$ there is no distinction between $I_q^n F$ and $I^n F$.
For background on quadratic and symmetric bilinear forms see \cite{EKM}.

\section{The Brauer Group}

For the standard case of $\operatorname{char}(k) \neq p$ the following is known:

\begin{prop}
Let $p$ be a prime integer and $k$ an algebraically closed field of $\operatorname{char}(k) \neq p$. Then $\operatorname{SymL}_p(F)=\lfloor \frac{n}{2} \rfloor$ for $F=k(\!(\alpha_1)\!) \dots (\!(\alpha_n)\!)$ being the field of iterated Laurent series in $n$ indeterminates.
\end{prop}

\begin{proof}
It follows from \cite[Theorem 7.80]{TignolWadsworth:2015}.
\end{proof}

We recall two theorems from the literature that are useful for our cause:
\begin{thm}[{\cite[Theorem 1]{Morandi:1989}}]\label{Morandi}
Suppose $F$ is a Henselian valued field, $D$ and $E$ are division algebras over $F$ such that 
\begin{enumerate}
\item $D$ is defectless,
\item $\overline{D} \otimes \overline{E}$ is a division algebra, and
\item $\Gamma_D \cap \Gamma_E=\Gamma_F$.
\end{enumerate}
Then $D \otimes E$ is a division algebra.
\end{thm}

\begin{thm}[{\cite[Theorem 3.3]{AravireJacob:1992}}]\label{Arav}
Let $F$ be a maximally complete field of $\operatorname{char}(F)=p>0$ with $\dim_{\mathbb{F}_p}(\Gamma_F/p\Gamma_F)=n$. Write $\dim_{\mathbb{F}_p}(\overline{F}/\wp(\overline{F}))=m$. Then the symbol length of ${_pBr}(F)$ is at most $n-1$ when $m<n$ and $n$ when $m \geq n$.
\end{thm}

We are now ready to prove the main results:

\begin{lem}\label{Div}
Let $p$ be a prime integer and $k$ a field of $\operatorname{char}(k)=p$. Then $[\alpha_2^{-1},\alpha_1)_{p,F} \otimes [\alpha_3^{-1},\alpha_2)_{p,F} \otimes \dots \otimes [\alpha_n^{-1},\alpha_{n-1})_{p,F}$ is a division algebra over $F=k(\!(\alpha_1)\!) \dots (\!(\alpha_n)\!)$.
\end{lem}

\begin{proof}
We prove this by induction.
For $n=2$, the algebra $[\alpha_2^{-1},\alpha_1)_{p,F}$ is indeed a division algebra:
The \'etale extension $K=F[x : x^p-x=\alpha_2^{-1}]$ is a field extension because the value of $\alpha_2^{-1}$ with respect to the right-to-left $(\alpha_1,\alpha_2)$-adic valuation is $(0,-1)$, and so if a solution to the equation $z^p-z=\alpha_2^{-1}$ existed in $F$, the root $z$ would be of value $(0,-\frac{1}{p})$, but $\Gamma_F=\mathbb{Z} \times \mathbb{Z}$, contradiction.
Now, the valuation extends to the field $K$ and $\Gamma_K=\mathbb{Z} \times \frac{1}{p} \mathbb{Z}$.
If $\alpha_1$ were a norm in $K/F$, since its value is $(1,0)$, there would exist an element in $K$ of value $(\frac{1}{p},0)$, but there is no such element. Hence $[\alpha_2^{-1},\alpha_1)_{p,F}$ is a division algebra.

Assume that $[\alpha_2^{-1},\alpha_1)_{p,F} \otimes [\alpha_3^{-1},\alpha_2)_{p,F} \otimes \dots \otimes [\alpha_m^{-1},\alpha_{m-1})_{p,F}$ is a division algebra for every $m \in \{2,\dots,n-1\}$ and any field $F$ of iterated Laurent series in $m$ indeterminates over any field $k$ of $\operatorname{char}(k)=p$.
Let us now look at $[\alpha_2^{-1},\alpha_1)_{p,F} \otimes [\alpha_3^{-1},\alpha_2)_{p,F} \otimes \dots \otimes [\alpha_n^{-1},\alpha_{n-1})_{p,F}$ over $F=k(\!(\alpha_1)\!) \dots (\!(\alpha_n)\!)$.
Write $F=K(\!(\alpha_n)\!)$ where $K=k(\!(\alpha_1)\!) \dots (\!(\alpha_{n-1})\!)$, $D=[\alpha_2^{-1},\alpha_1)_{p,F} \otimes [\alpha_3^{-1},\alpha_2)_{p,F} \otimes \dots \otimes [\alpha_{n-1}^{-1},\alpha_{n-2})_{p,F}$ and $E=[\alpha_n^{-1},\alpha_{n-1})_{p,F}$.
Consider the $\alpha_n$-adic valuation on $F$.
Then $\overline{E}=K[q : q^p=\alpha_{n-1}]=k(\!(\alpha_1)\!) \dots (\!(\alpha_{n-2})\!)(\!(q)\!)$.
In addition, the algebra $D$ is defectless with respect to this valuation, and we have $\overline{D} \otimes \overline{E}=[\alpha_2^{-1},\alpha_1)_{p,\overline{E}} \otimes [\alpha_3^{-1},\alpha_2)_{p,\overline{E}} \otimes \dots \otimes [\alpha_{n-2}^{-1},\alpha_{n-3})_{p,\overline{E}} \otimes [q^{-1},\alpha_{n-2})_{p,\overline{E}}$, which is a division algebra by the induction hypothesis.
The conditions of Theorem \ref{Morandi} are therefore satisfied, and so $D\otimes E$ is a division algebra over $F$, and the induction is complete.
\end{proof}

\begin{thm}
Let $p$ be a prime integer, $k$ a perfect field of $\operatorname{char}(k)=p$ and $F=k(\!(\alpha_1)\!) \dots (\!(\alpha_n)\!)$. Then $\operatorname{SymL}_p(F)=n-1$ when $\dim_{\mathbb{F}_p}(k/\wp(k)) < n$ (e.g., when $k$ is algebraically closed), and $\operatorname{SymL}_p(F)=n$ when $\dim_{\mathbb{F}_p}(k/\wp(k)) \geq n$.
\end{thm}

\begin{proof}
Suppose $\dim_{\mathbb{F}_p}(k/\wp(k)) < n$.
Consider the right-to-left $(\alpha_1,\dots,\alpha_n)$-adic valuation on $F$.
Since the residue field is algebraically closed, by Theorem \ref{Arav} the symbol length of ${_pBr}(F)$ is at most $n-1$.
On the other hand, by Lemma \ref{Div} there exists a class of symbol length $n-1$ in ${_pBr}(F)$, so the symbol length of ${_pBr}(F)$ is exactly $n-1$.

Now suppose $\dim_{\mathbb{F}_p}(k/\wp(k)) \geq n$.
By Theorem \ref{Arav} the symbol length of ${_pBr}(F)$ is at most $n$. Take $\beta_1,\dots,\beta_n \in k$ to be $\mathbb{F}_p$-linearly independent elements in $k/\wp(k)$, and consider the algebra $D=[\beta_1,\alpha_1)_{p,F} \otimes \dots \otimes [\beta_n,\alpha_n)_{p,F}$.
This algebra is generated over $F$ by $x_1,y_1,\dots,x_n,y_n$ satisfying $x_i^p-x_i=\beta_i$, $y_i^p=\alpha_i$ and $y_i x_i y_i^{-1}=x_i+1$, $y_i y_j=y_j y_i$ and $x_i x_j=x_j x_i$ for any $i,j \in \{1,\dots,n\}$ with $i\neq j$. Hence, this algebra can also be viewed as the skew field of twisted iterated Laurent series $L(\!(y_1;\sigma_1)\!)\dots(\!(y_n;\sigma_n)\!)$ over $L=k[x_1,\dots,x_n : x_i^p-x_i=\beta_i \forall i\in \{1,\dots,n\}]$ where each $\sigma_i$ is the automorphism of $L$ mapping $x_i$ to $x_i+1$ and every other $x_j$ to itself (see \cite[Section 1.1.3]{TignolWadsworth:2015}). In particular, $D$ is a division algebra whose symbol length is $n$, and so $\operatorname{SymL}_p(F)=n$.
\end{proof}

\begin{cor}
When $\operatorname{char}(k)=2$, the field $F=k(\!(\alpha)\!)(\!(\beta)\!)(\!(\gamma)\!)$ is not linked.
\end{cor}

Note that this means that $I_q^2 F$ is not linked for $F=k(\!(\alpha)\!)(\!(\beta)\!)(\!(\gamma)\!)$. In the next section we generalize this fact to arbitrary quadratic and bilinear Pfister forms.

\section{Quadratic and Bilinear Forms}

We first cover the standard case:
\begin{prop}
Given an algebraically closed field $k$ of $\operatorname{char}(k) \neq 2$, $I^n F$ is linked for $F=k(\!(\alpha_1)\!)\dots(\!(\alpha_{n+1})\!)$.
\end{prop}

\begin{proof}
Consider the $n$-fold Pfister forms $\varphi=\langle\! \langle a_1,\dots,a_n \rangle\!\rangle$ and $\psi=\langle\! \langle b_1,\dots,b_n \rangle\!\rangle$. By Hensel's Lemma, the $a$-s and $b$-s can be assumed to be products of powers of the $\alpha$-s.
Consider the group $F^\times/(F^\times)^2$ as a vector space $V \cong \mathbb{F}_2^{\times (n+1)}$.
By the identity $\langle\!\langle a,b \rangle\!\rangle=\langle\!\langle ab,b \rangle\!\rangle$, and the assumption that $\varphi$ and $\psi$ are anisotropic, the elements $a_1,\dots,a_n$ span an $n$-dimensional subspace $W_a$ of $V$ and the elements $b_1,\dots,b_n$ span an $n$-dimensional subspace $W_b$ of $V$.
Therefore, $W_a \cap W_b$ is at least of dimension $n-1$.
Take a basis $c_1,\dots,c_{n-1}$ for this $(n-1)$-dimensional subspace of the intersection, and $\langle\!\langle c_1,\dots,c_{n-1} \rangle\!\rangle$ is a common factor of $\varphi$ and $\psi$.
\end{proof}

When $\operatorname{char}(k)=2$ we need to distinguish between quadratic and bilinear forms.

\begin{thm}\label{quadratic}
Given a field $k$ of $\operatorname{char}(k)=2$, $I_q^n F$ is not linked for $F=k(\!(\alpha_1)\!)\dots(\!(\alpha_{n+1})\!)$.
\end{thm}

\begin{proof}
Let $\varphi=\langle\!\langle \alpha_1,\dots,\alpha_{n-2},\alpha_{n-1},\alpha_n^{-1}]\!]$ and $\psi=\langle\!\langle \alpha_1,\dots,\alpha_{n-2},\alpha_n,\alpha_{n+1}^{-1}]\!]$.
The form $\varphi \perp \psi$ is Witt equivalent to $\omega=\langle\!\langle \alpha_1,\dots,\alpha_{n-2}\rangle\!\rangle \otimes [1,\alpha_n^{-1}+\alpha_{n+1}^{-1}] \perp \alpha_{n-1} \langle\!\langle \alpha_1,\dots,\alpha_{n-2}\rangle\!\rangle \otimes [1,\alpha_n^{-1}] \perp \alpha_n \langle\!\langle \alpha_1,\dots,\alpha_{n-2}\rangle\!\rangle \otimes [1,\alpha_{n+1}^{-1}]$.
For $\varphi$ and $\psi$ to share a common $(n-1)$-fold factor, $\omega$ must be isotropic, but it is not.
This can be seen by considering the right-to-left $(\alpha_1,\dots,\alpha_{n+1})$-adic valuation. The values modulo $(2\mathbb{Z})^{\times (n+1)}$ of all the symplectic blocks in $\omega$ are distinct, and so the form must be anisotropic.
\end{proof}

\begin{thm}\label{bilinear}
Let $F$ be a field of $\operatorname{char}(F)=2$ and $\operatorname{rank}_2(F) \geq n+1$. Then $I^n F$ is not linked.
\end{thm}

\begin{proof}
Let $\langle\!\langle\alpha_1,\dots,\alpha_{n+1}\rangle\!\rangle$ be an anisotropic bilinear $n$-fold Pfister form over $F$.
Consider the forms $\varphi=\langle \! \langle \alpha_1,\dots,\alpha_{n-2},\alpha_{n-1},\alpha_n\rangle\!\rangle$ and $\psi=\langle \! \langle \alpha_1,\dots,\alpha_{n-2},\alpha_{n-1}+1,\alpha_{n+1}\rangle\!\rangle$.
Recall that an anisotropic bilinear Pfister form $\tau$ has a ``pure subform" $\tau'$ which is the unique anisotropic symmetric bilinear form satisfying $\tau'\perp \langle 1 \rangle=\tau$ (see \cite[Chapter 1, Section 6]{EKM}).
We therefore denote by $\varphi'$ and $\psi'$ the pure subforms of $\varphi$ and $\psi$, respectively.
The set $D(\varphi') \cup \{0\}$ is the $(2^n-1)$-dimensional $F^2$-vector space
$$V_\varphi=\bigoplus_{(e_1,\dots,e_n)\in \{0,1\}^{\times n} \setminus \{(0,\dots,0)\}} F^2 \alpha_1^{e_1} \dots \alpha_n^{e_n}.$$
The set $D(\psi') \cup \{0\}$ is the $(2^n-1)$-dimensional $F^2$-vector space
$$V_\psi=\bigoplus_{(e_1,\dots,e_n)\in \{0,1\}^{\times n} \setminus \{(0,\dots,0)\}} F^2 \alpha_1^{e_1} \dots \alpha_{n-2}^{e_{n-2}} (\alpha_{n-1}+1)^{e_{n-1}} \alpha_{n+1}^{e_{n+1}}.$$
In order for $\psi$ and $\varphi$ to share an $(n-1)$-fold factor, the intersection $V_\varphi \cap V_\psi$ should be of dimension at least $2^{n-1}-1$.
However, the intersection of these two spaces is
$W \oplus \alpha_{n-1} W$ where $W=\bigoplus_{(e_1,\dots,e_{n-2})\in \{0,1\}^{\times (n-1)} \setminus \{(0,\dots,0)\}} F^2 \alpha_1^{e_1} \dots \alpha_{n-2}^{e_{n-2}}$, which is of dimension $(2^{n-2}-1)\cdot 2=2^{n-1}-2$, i.e., less than $2^{n-1}-1$.
Therefore, the forms do not share an $(n-1)$-fold factor.
\end{proof}

\begin{cor}
Given a field $F$ of $\operatorname{char}(F)=2$ with $I^n F \neq 0$, the following are equivalent:
\begin{enumerate}
\item $I^n F$ is linked.
\item $I^n F$ is 3-linked.
\item The 2-rank of $F$ is $n$.
\end{enumerate}
\end{cor}

\begin{proof}
By the assumption, the 2-rank of $F$ is at least $n$.
By Theorem \ref{bilinear}, (1) implies (3).
By \cite[Corollary 3.2]{Chapman:2018}, (3) implies (2), and clearly (2) implies (1).
\end{proof}

Note that when $k$ is an algebraically closed field of $\operatorname{char}(k)=2$, the field $F=k(\!(\alpha_1)\!)\dots(\!(\alpha_{n+1})\!)$ satisfies the conditions of Theorems \ref{quadratic} and \ref{bilinear}.

\section*{Acknowledgements}

The author thanks Jean-Pierre Tignol for the helpful discussions, and the anonymous referee for the constructive suggestions.
\bibliographystyle{abbrv}

\end{document}